\documentclass[12pt]{amsart} 
\usepackage{amssymb,amsmath,latexsym,enumerate,graphicx,bbm,mathptmx,microtype}
\newtheorem{theorem}{Theorem} 

\newtheorem{lemma}[theorem]{Lemma}

\newtheorem*{rem}{Remarks}

\newtheorem*{exam}{Examples}

\def\Z{\mathbbm{Z}}

\def\g{g^*}
\def\n{n^*}

\title{An Extreme Family of Generalized Frobenius Numbers}

\author{Matthias Beck}
\address{Department of Mathematics\\
         San Francisco State University\\
         San Francisco, CA 94132\\
         U.S.A.}
\email{[matthbeck,curtis.kifer]@gmail.com}

\author{Curtis Kifer}

\subjclass[2000]{11D07, 05A17}
\keywords{Linear Diophantine problem of Frobenius, coin-exchange problem, Frobenius number.}
\thanks{Research partially supported by the NSF (research grant DMS-0810105). We thank the anonymous referee for helpful suggestions to an earlier version of this paper.}
\date{27 January 2011}

\begin{document}

\begin{abstract}
We study a generalization of the \emph{Frobenius problem}: given $k$ positive relatively prime integers, what is the largest integer $g_0$ that cannot be represented as a nonnegative integral linear combination of the given integers? More generally, what is the largest integer $g_s$ that has exactly $s$ such representations?
We construct a family of integers, based on a recent paper by Tripathi, whose generalized Frobenius numbers $g_0, \ g_1, \ g_2, \dots$ exhibit unnatural jumps; namely, $g_0, \ g_1, \ g_k, \ g_{ \binom{ k+1 }{ k-1 } } , \ g_{ \binom{ k+2 }{ k-1 } } , \dots$ form an arithmetic progression, and any integer larger than $g_{ \binom{ k+j }{ k-1 } }$ has at least $\binom{ k+j+1 }{ k-1 }$ representations. Along the way, we introduce a variation of a generalized Frobenius number and prove some basic results about it.
\end{abstract}

\maketitle


\section{Introduction}

Given positive relatively prime integers $a_1, a_2, \dots, a_k$ (i.e., $\gcd(a_1, a_2, \dots, a_k) = 1$),
we call $x \in \Z_{ \ge 0 }$ \emph{representable} if 
\begin{equation}\label{repequ}
  x = m_1 a_1 + m_2 a_2 + \dots + m_k a_k
\end{equation} 
for some $m_1, m_2, \dots, m_k \in \Z_{ \ge 0 }$.
The \emph{linear Diophantine problem of Frobenius} asks for the largest integer $g = g \left( a_1, a_2, \dots, a_k \right)$ (the \emph{Frobenius number}) that is not representable; this is arguably one of the most famous problems in combinatorial number theory.
Its study was initiated in the 19th century, and it has been known at least since 1884 that 
\begin{equation}\label{gfor2}
  g(a_{1}, a_{2}) = a_{1} a_{2} - a_1 - a_2 \, .
\end{equation} 
(Sylvester's paper \cite{sylvester} gives a clear indication that he knew \eqref{gfor2}.)
One fact that makes the Frobenius problem attractive is that it can be easily
described, for example, in terms of coins of denominations $a_1, a_2, \dots, a_k$; the 
Frobenius number is the largest amount of money that cannot be formed using these coins. 
For details about the Frobenius problem, including numerous applications, see~\cite{ramirezbook}.

Recent work on the Frobenius problem includes what some authors call the \emph{$s$-Frobenius number} $g_s = g_s \left( a_1, a_2, \dots, a_k \right)$: the largest integer that has exactly $s$ representations. 
(To say that \emph{$x$ has $s$ representations} means that \eqref{repequ} has $s$ solutions $\left( m_1, m_2, \dots, m_k \right) \in \Z_{ \ge 0 }^k$.)
Note that $g = g_0$.
For example, Beck--Robins \cite{frobnote} showed that
\begin{equation}\label{gsfor2}
  g_s(a_1, a_2) = (s+1) a_1 a_2 - a_1 - a_2 \, ,
\end{equation} 
which generalizes \eqref{gfor2}.
Inspired by computational evidence in \cite{brownetalfrobenius}, Shallit and Stankewicz \cite{shallitstankewiczfrobenius} constructed an ``extreme" family of 5-tuples, namely one for which $g_0 - g_s$ is unbounded (for any $s \ge 1$), as well as examples with $g_0 > g_1$ for $k \ge 6$, contrary to the ``natural" order $g_0 < g_1 < g_2 < \cdots$ one might expect.

The goal of this paper is to exhibit another ``extreme" family of $k$-tuples, yet whose behavior differs from that described in \cite{shallitstankewiczfrobenius}. The easiest way to state our main result is to introduce a variation of the $s$-Frobenius number, which might be of independent interest:
we define $\g_s = \g_s \left( a_1, a_2, \dots, a_k \right)$ to be the largest integer that has \emph{at most} $s$ representations.
Note that $\g_0 = g_0 = g$ and $g_s \le \g_s$.
Our main result is as follows.

\begin{theorem}\label{maintheorem}
Let $a_{1},a_{2}, \dots,a_{k}$ be pairwise relatively prime positive integers,
\[
  \Pi := a_1 a_2 \cdots a_k \, , \qquad
  A_{j}:=\tfrac{\Pi}{a_{j}} \, , \qquad
  \Sigma := A_1 + A_2 + \dots + A_k \, , 
\]
and let $g_s  = g_s \left( A_1, A_2, \dots, A_k \right)$ and $\g_s  = \g_s \left( A_1, A_2, \dots, A_k \right)$.
Then for $t \ge 1$,
\[
  g_{ \binom{ k+t-2 }{ k-1 } } = \g_{ \binom{ k+t-2 }{ k-1 } } = \g_{ \binom{ k+t-2 }{ k-1 } + 1 } = \g_{ \binom{ k+t-2 }{ k-1 } + 2 } = \dots = \g_{ \binom{ k+t-1 }{ k-1 } - 1 } 
  = (k+t-1) \Pi - \Sigma \, .
\]
In words: the integer $(k+t-1) \Pi - \Sigma$ has exactly $\binom{ k+t-2 }{ k-1 }$ representations, and every larger integer has at least $\binom{ k+t-1 }{ k-1 }$ representations.
\end{theorem}

This theorem extends \eqref{gsfor2}, which is the case $k=2$ (where we have $A_1 = a_2$ and $A_2 = a_1$).
Theorem \ref{maintheorem} also generalizes---and was motivated by---Tripathi's theorem \cite{tripathifrobeniusextension}, which is the case $t=0$:
\begin{equation}\label{tripathieq}
  g_0 = (k-1) \Pi - \Sigma \, .
\end{equation}
For this specific family of $k$-tuples, Theorem \ref{maintheorem} gives
\begin{align*}
  g_1 &= k \, \Pi - \Sigma, \text{ beyond which all integers have at least $k = \tbinom{ k }{ k-1 } $ representations; } \\
  g_k &= (k+1) \Pi - \Sigma, \text{ beyond which all integers have at least $\tbinom{ k+1 }{ k-1 } $ representations; } \\
  g_{ \binom{ k+1 }{ k-1 } } &= (k+2) \Pi - \Sigma, \text{ beyond which all integers have at least $\tbinom{ k+2 }{ k-1 } $ representations; etc. }
\end{align*}

We prove Theorem \ref{maintheorem} in Section \ref{mainproofsection}. 
Returning to the general case, in Section \ref{lemmasection} we extend a well-known lemma on the Frobenius number due to Brauer and Shockley \cite{brauershockley}, which has been successfully applied in numerous proofs dealing with certain instances of the Frobenius problem. Our generalization gives a lemma for $\g_s$; we do not know if a similar lemma exists for $g_s$. This might be an indication that $\g_s$ (compared with $g_s$) is the more useful generalized Frobenius number.


\section{Proof of Theorem \ref{maintheorem}}\label{mainproofsection}

We claim that for $t \ge 1$, all representations of $(k+t-1) \Pi - \Sigma$ are of the form
\begin{equation}\label{star3}
  \left( (n_1 + 1) a_1 - 1 \right) A_1 + \left( (n_2 + 1) a_2 - 1 \right) A_2 + \dots + \left( (n_k + 1) a_k - 1 \right) A_k \, ,
\end{equation}
where $n_1, n_2, \dots, n_k$ are nonnegative integers that sum to $t-1$.
This gives the desired number of $\binom{ k+t-2 }{ k-1 }$ representations, and we conclude Theorem \ref{maintheorem} by proving that every integer larger than $(k+t-1) \Pi - \Sigma$ has at least $\binom{ k+t-1 }{ k-1 }$ representations.
It will be handy to remember that the coefficient of $A_j$ in \eqref{star3} satisfies 
\begin{equation}\label{starf}
  (n_j + 1) a_j - 1 \equiv -1 \pmod{a_j} \, .
\end{equation}

\begin{proof}[Proof of Theorem \ref{maintheorem}]
We use induction on $t$. 
In the case $t=1$, note that $k \, \Pi - \Sigma$ is representable:
\[
  k \, \Pi - \Sigma = \left( a_1 - 1 \right) A_1 + \left( a_2 - 1 \right) A_2 + \dots + \left( a_k - 1 \right) A_k \, .
\]
Any other representation
\[
  x_1 A_1 + x_2 A_2 + \dots + x_k A_k
\]
would necessarily contain an $x_j$ that is at least $a_j$, and thus
\[
  (k-1) \Pi - \Sigma =
  x_1 A_1 + x_2 A_2 + \dots + \left( x_j - a_j \right) A_j + \dots + x_k A_k
\]
would be representable, contradicting~\eqref{tripathieq}.

Now we prove that if $n > k \, \Pi - \Sigma$, then $n$ has at least $k$ representations.
Since $n > g_0 + \Pi$, there is a representation
$
  x_1 A_1 + x_2 A_2 + \dots + x_k A_k
$
of $n - \Pi$. To get a representation of $n$, we can add $a_j$ to $x_j$ for any $j$. Thus $n$ has at least $k$ representations.
This finishes the base case.

Now assume that $t>1$ and $(k+t-2) \Pi - \Sigma$ has exactly the $\binom{ k+t-3 }{ k-1 }$ representations realized by \eqref{star3} with $n_1 + n_2 + \dots + n_k = t-2$.

To any of these representations, we can add $\Pi = a_j \, A_j$, which gives a representation of $(k+t-1) \Pi - \Sigma$ of the form \eqref{star3} with $n_1 + n_2 + \dots + n_k = t-1$. This accounts for $\binom{ k+t-2 }{ k-1 }$ representations of $(k+t-1) \Pi - \Sigma$.
On the other hand, any representation
\begin{equation}\label{newrepeq}
  x_1 A_1 + x_2 A_2 + \dots + x_k A_k
\end{equation}
of $(k+t-1) \Pi - \Sigma$ that is \emph{not} of this form will have an $x_j \not\equiv -1 \pmod{a_j}$. So then $x_j = q a_j + r$ for some integers $q \ge 0$ and $0 \le r \le a_j - 2$.
If $q>0$ then 
\[
  x_1 A_1 + x_2 A_2 + \dots + \left( (q-1) a_j + r \right)  A_j + \dots + x_k A_k
\]
is a representation of $(k+t-2) \Pi - \Sigma$ that violates \eqref{starf} and the induction hypothesis.
If $q=0$ then $x_j = r$ and so there must be an $x_i \ge a_i$ (otherwise \eqref{newrepeq} would be $< k\, \Pi - \Sigma$). 
But then
\[
  x_1 A_1 + x_2 A_2 + \dots + (x_i - a_i) A_i + \dots + r A_j + \dots + x_k A_k
\]
is a representation of $(k+t-2) \Pi - \Sigma$ that violates \eqref{starf} and the induction hypothesis.
Thus no ``new" representation of the form \eqref{newrepeq} can exist for $(k+t-1) \Pi - \Sigma$, and so there are exactly $\binom{ k+t-2 }{ k-1 }$ representations of $(k+t-1) \Pi - \Sigma$.

Finally, we prove that if $n > (k+t-1) \Pi - \Sigma$, then $n$ has at least $\binom{ k+t-1 }{ k-1 }$ representations.
Since $n > g_0 + t \Pi$, there is a representation
$
  x_1 A_1 + x_2 A_2 + \dots + x_k A_k
$
of $n - t \Pi$. To get a representation of $n$, we can add $a_j$ to $x_j$ for any $j$, a total of $t$ times. Thus $n$ has at least $\binom{ k+t-1 }{ k-1 }$ representations.
\end{proof}


\section{A Brauer--Shockley Lemma for $\g_s$}\label{lemmasection}

Experts on the Frobenius problem will recognize the $s=0$ case of the following lemma from many papers on the Frobenius problem, starting with~\cite{brauershockley}.

\begin{lemma}\label{lemma1}
Let $a_{1},a_{2}, \dots, a_{k}$ be positive integers with  $\gcd (a_{1},a_{2}, \dots,a_{k})=1$, and let $n_{j,s}$ denote the least nonnegative integer congruent to $j \pmod{a_1}$ that has more than $s$ representations. Then 
\[
  \g_{s}(a_{1},a_{2}, \dots, a_{k}) = \max_{0 \le j \le a_1-1} n_{j,s}-a_1 \, .
\]
\end{lemma}

We included this result here because the $s=0$ case proved useful in many special instances of the Frobenius problem, and because we believe it gives credence to the definition of $\g_s$, in that we are unable to prove a version of Lemma \ref{lemma1} for~$g_s$.
It is an interesting open problem to study in what instances $g_s < \g_s$.

\begin{proof}
The number $\max_{0 \le j \le a_1-1} n_{j,s}-a_1$ has at most $s$ representations, since each $n_{j,s}-a_1$ has by definition at most $s$ representations.

Now let $x > \max_{0 \le j \le a_1-1} n_{j,s}-a_1$ and, say, $x \equiv m \pmod{a_1}$. 
Then $x \equiv n_{m,s} \pmod{a_1}$ and thus $x \ge n_{m,s}$. So $x$ has more than $s$ representations by definition.
\end{proof}

In the literature on the Frobenius problem, the number $n \left( a_1, a_2, \dots, a_k \right)$ of nonrepresentable integers plays a role almost as famous as that of $g \left( a_1, a_2, \dots, a_k \right)$.
In analogy with our definition of $\g_s$, we denote by  $\n_{s}(a_{1},a_{2}, \dots,a_{k})$ the number of nonnegative integers that have at most $s$ representations.
The $s=0$ case of the following result is due to Selmer \cite[p.~3]{selmer}. We leave the proof, which is similar to that of Lemma \ref{lemma1}, to the reader.

\begin{lemma}\label{lemma2}
With the same notation and conditions as in Lemma \ref{lemma1},
\[
  \n_{s}(a_{1},a_{2}, \dots,a_{k}) = \frac{1}{a_1}\sum _{j=0}^{a_{1}-1}n_{j,s}-\frac{a_{1}-1}{2} \, .
\]
\end{lemma}

As an illustration of the applicability of Lemma \ref{lemma1}, we include a last result, which generalizes well-known theorems of Johnson \cite[Theorem 2]{johnson} and R{\o}dseth~\cite[Lemma 1]{rodseth1}.

\begin{lemma}
Let $a_{1},a_{2}, \dots, a_{k}$ be positive integers with  $\gcd (a_{1},a_{2}, \dots,a_{k})=1$.
If  $\gcd(a_{2},a_{3}, \dots,a_{k})=d$, let  $a_{j}=d\, a_{j}'$ for $2 \le j \le k$. Then
\begin{align*}
  \g_{s}(a_{1}, a_{2}, \dots, a_{k}) &= d \, \g_{s}(a_{1},a_{2}',a_{3}', \dots,a_{k}')+a_{1}(d-1) \, , \\
  \n_{s}(a_{1}, a_{2}, \dots, a_{k}) &= d \, \n_{s}(a_{1},a_{2}',a_{3}', \dots,a_{k}')+\tfrac{1}{2}(a_{1}-1)(d-1) \, .
\end{align*}
\end{lemma}

This lemma does have a counterpart for $g_s$ \cite{brownetalfrobenius}.

\begin{proof}
As in Lemma \ref{lemma1},
let $n_{j,s}$ denote the least nonnegative integer congruent to $j \pmod{a_1}$ that has more than $s$ representations in terms of $a_{1}, a_{2}, \dots, a_{k}$, and similarly
let $n_{j,s}'$ denote the least nonnegative integer congruent to $j \pmod{a_1}$ that has more than $s$ representations in terms of $a_{1},a_{2}',a_{3}', \dots,a_{k}'$. 
It is a fun exercise to check that
\[
  \left\{ n_{ j,s } : 0 \le j \le a_1-1 \right\} = \left\{ d \, n_{ j,s }' : 0 \le j \le a_1-1 \right\}
\]
and so by Lemma \ref{lemma1},
\begin{align*}
  g_{s}(a_{1},a_{2}, \dots,a_{k})
  &= \max_{0 \le j \le a_1-1}{n_{j,s}-a_1}
   = \max_{0 \le j \le a_1-1}d\, n_{j,s}'-a_{1} 
   = d \biggl( \max_{0 \le j \le a_1-1}n_{j,s}'-a_{1} \biggr) + a_{1}(d-1) \\
  &= d\, g_{s}(a_{1},a_{2}',a_{3}', \dots,a_{k}')+a_{1}(d-1) \, ,
\end{align*}
and by Lemma \ref{lemma2},
\begin{align*}
  n_{s}(a_{1},a_{2}, \dots,a_{k})
  &= \frac{1}{a_{1}}\sum _{j=0}^{a_{1}-1}n_{{j,s}}-\frac{a_{1}-1}{2}
   = d \left( \frac{1}{a_{1}}\sum _{j=0}^{a_{1}-1}n_{{j,s}}'-\frac{a_{1}-1}{2} \right) + \frac{1}{2}(a_{1}-1)(d-1)\\
  &= d \, n_{s}(a_{1},a_{2}',a_{3}', \dots,a_{k}')+\frac{1}{2}(a_{1}-1)(d-1). \qedhere
\end{align*}
\end{proof}


\bibliographystyle{amsplain}

\def\cprime{$'$} \def\cprime{$'$}
\providecommand{\bysame}{\leavevmode\hbox to3em{\hrulefill}\thinspace}
\providecommand{\MR}{\relax\ifhmode\unskip\space\fi MR }
\providecommand{\MRhref}[2]{%
  \href{http://www.ams.org/mathscinet-getitem?mr=#1}{#2}
}
\providecommand{\href}[2]{#2}

\setlength{\parskip}{0cm} 
\end{document}